\documentclass[11pt]{amsart}

\usepackage{latexsym,amsfonts,amssymb,epsfig,verbatim}
\usepackage{amsmath,amssymb,latexsym,graphics,textcomp,hyperref,enumerate} 
\usepackage[foot]{amsaddr}
\usepackage[utf8]{inputenc}
\usepackage{amssymb}
\usepackage{comment}
 \usepackage[T1]{fontenc}
 \usepackage{lmodern}
 \usepackage{graphicx} 
\usepackage{xspace} 
  
\usepackage[usenames,dvipsnames]{xcolor}

\usepackage{tikz}
\usepackage{pgfplots}

\usepackage{bm}
\usepackage{amsfonts}
\usepackage{amsmath}
\usepackage{enumerate}
\usepackage{eucal}
\usepackage{amsthm}
\usepackage{thmtools}
\usepackage[capitalize]{cleveref}
\newtheorem{theorem}{Theorem}[section]
\newtheorem{proposition}[theorem]{Proposition}
\newtheorem{lemma}[theorem]{Lemma}
\newtheorem{corollary}[theorem]{Corollary}
\theoremstyle{definition}
\newtheorem{definition}[theorem]{Definition}

\newcommand{\teich}{\mathcal{T}}

\newcommand{\ml}{\mathcal{ML}}
\newcommand{\pml}{\mathcal{PML}}
\newcommand{\reals}{\mathbb{R}}
\newcommand{\naturals}{\mathbb{N}}

\renewcommand{\th}{\mathrm{Th}}

\newcommand{\chr}{\mathcal{CR}}
\newcommand{\ie}{i.e.\ }
\newcommand{\mute}[1] {}

\newcommand{\Fr}{\mathrm{Fr}}

\newcommand{\MF}{\ensuremath{\mathcal{MF}}\xspace}

\renewcommand{\v}{\mathbf v}

\newcommand{\cat}{\operatorname{cat}}
\newcommand{\Int}{\mathrm{Int}}
\newcommand{\codim}{\mathrm{codim}}

\title[Teichm\"{u}ller space with Thurston's Finsler metric]{The combinatorial structure of the unit tangent spheres and cotangent spheres of Teichm\"{u}ller space with Thurston's Finsler metric}
\author{Ken'ichi Ohshika}

\author{Athanase Papadopoulos}
\address{KO: Department of Mathematics, Gakushuin University, Mejiro, Toshima-ku, 171-8588 Tokyo, Japan, and Max-Planck Institute for Mathematics, Vivatsgasse 7, 53111 Bonn
Germany,  ORCID: 0000-0001-5822-7425}
\email{ohshika@math.gakushuin.ac.jp ;  papadop@math.unistra.fr}
\address{AP: Institut de Recherche Mathématique Avancée, 
 CNRS et Universit\'e de Strasbourg, 
7 rue Ren\'e Descartes, 67084, Strasbourg, France, and Max-Planck Institute for Mathematics, Vivatsgasse 7, 53111 Bonn
Germany,  ORCID:  0000-0002-0880-0307}
\date{}

\begin{document}
\maketitle
\begin{abstract}

We prove several new results on the combinatorial structures of the unit spheres of the norms induced by Thurston's metric  on the tangent and cotangent spaces of the Teichm\"uller space of  a closed surface of negative Euler characteristic.
These results include a formula for the dimension of every face of a unit sphere in the tangent space in terms of an    invariant of the chain-recurrent lamination representing the face.  We then prove that the combinatorial structure of such a unit sphere is independent of the underlying point in Teichm\"uller space. Provided the genus of the surface is $\geq 2$, we show that there is a natural isomorphism between the extended mapping class group of the surface and the group of combinatorial automorphisms of such a unit sphere. In the case of genus 2, we obtain a natural epimorphism between the two groups whose kernel is the class of the hyperelliptic involution. Regarding the unit spheres of Thurston's metric in the cotangent spaces, we obtain a formula describing the codimensions of faces of such a sphere in terms of corresponding projective measured laminations. We then give a necessary and sufficient condition for a face  to be exposed, and of a face to correspond to a projectively weighted multi-curve.  Some of the results obtained answer open questions.



\end{abstract}

\section{Introduction}
Thurston's asymmetric metric on Teichm\"{u}ller spaces of a hyperbolic surface $S$  is a natural analogue of Teichm\"{u}ller's more classical metric whose setting is conformal geometry, with Lipschitz maps, in the hyperbolic setting, replacing the quasi-conformal maps of the conformal setting.
Thurston's  metric, like Teichm\"uller's,  is Finsler. This means that it is induced by a family of norms on tangent spaces to Teichm\"{u}ller space.   The unit spheres of Thurston's norms have a (non-locally finite) polyhedral-like structure, whereas  the unit spheres of these norms are smooth in Teichm\"uller's setting, and this gives rise to  new kinds of questions regarding Thurston's  metric.
In this paper, we work with the unit spheres with respect to the norms corresponding to Thurston's metric on tangent spaces and cotangent spaces.

 In \cite{BOP}, we gave a geometric interpretation of the faces of the unit spheres in the tangent spaces.
In particular, we showed that each face of a unit tangent sphere corresponds to a chain-recurrent geodesic lamination which is infinitesimally maximally stretched by the vectors lying on that face.
In the present paper, we give a more complete description of these unit spheres. We start by establishing a formula for the dimension of a face.
This result is  closely related to a result of Thurston's in  \cite[Theorem 10.1]{ThM}, where he showed that almost all points in a unit tangent sphere lie on facets, \ie faces of the highest possible dimension, and that these facets are in natural one-to-one correspondence with the simple closed geodesics on the surface.

Our first result, \cref{dimension}, is a generalisation of Thurston's result to faces of any dimension, giving a formula for the dimension of such a face in terms of a numerical value associated with a  chain-recurrent lamination, which we call its complexity. Our argument is along the same line as Thurston's, involving the differential of the cataclysm coordinates of Teichm\"uller space introduced by Thurston in the same paper.

We apply this theorem to prove two remarkable properties of the unit tangent sphere.
The first property concerns the invariance of the combinatorial type of the unit tangent sphere with respect to the base point (\cref{combinatorially invariant}).
To be more precise, we prove that for any two points $x$ and $y$ in 
$\teich(S)$, there is a combinatorial automorphism (\ie a homeomorphism preserving the face structures) between the unit tangent spheres $\mathcal S_x$ and  $\mathcal S_y$ with respect to the norms corresponding to Thurston's metric which preserves the natural combinatorial structure induced by a labelling of the points of these spheres by chain-recurrent geodesic laminations.  This result answers a question raised by Norbert A'Campo in 2022.
The second application, \cref{combinatorial auto},  is a rigidity result which says the following: for any point $x\in \mathcal{T}(S)$, there is a natural epimorphism from the mapping class group of the surface
to the  group of combinatorial automorphisms of the unit sphere $\mathcal S_x$. 
In the case when $S$ has genus $2$, this epimorphism is not injective with kernel generated by the hyperelliptic involution, and otherwise the epimorphism is in fact an isomorphism.
 This result, which relies on the famous Ivanov theorem \cite{Iv},  is in the spirit of several  rigidity results concerning mapping class group actions on spaces of geodesic laminations, obtained in the last few years, see \cite{CPP, Ohshika2013, OP-CRAS, OP2019, Pap2008}. These results are formally similar to each other, but they are different in content and they use different features of the spaces of measured laminations. 

In the last section, we turn to the unit cotangent spheres with respect to Thurston's metric.
In \cite{HOP}, the notion of codimension for faces of convex spheres was introduced, and asked how we can determine the codimension of faces of unit cotangent spheres.
In \cref{codimension of face}, we present a formula describing the codimensions of faces in terms of corresponding projective measured laminations, which gives an answer to the question.
We also give a necessary and sufficient condition on a face to be exposed  (\cref{exposed}), and  for a face to correspond to a projectively weighted multi-curve, in terms of the sum of dimension and codimension (\cref{sum}).

This paper was written during the authors' visit to the Max Planck Institute for Mathematics in Bonn.
They are deeply grateful to the MPIM for its hospitality and support throughout their stay.

\section{Preliminaries}
\subsection{Generalities}
Throughout this paper, we denote by $S$ an oriented closed surface of genus $\geq 2$.
Such a surface $S$ admits hyperbolic metrics. The Teichm\"{u}ller space  $\teich(S)$ of $S$ is  the set of hyperbolic metrics on $S$ modulo isotopy.
We endow $\teich$ with  the quotient topology of the space of hyperbolic metrics on $S$.

As we shall deal with unit spheres in tangent spaces and cotangent spaces at points in Teichm\"uller space equipped with the Thurston metric, and as these spheres are convex spheres in their respective ambient spaces, we recall here some terminology of convex geometry which will be used later.

Let $V$ be a finite-dimensional vector space.
A closed convex subset $\mathcal{C}$ of $V$ is homeomorphic to a geometric ball in $V$ and is called a convex ball. 
 A proper subset $F$ of $\mathcal{C}$ is called a \emph{face} 
of $\mathcal{C}$ if any segment in $\mathcal{C}$ intersecting $F$ at an interior point is entirely contained in $F$. A face of $\mathcal{C}$ necessarily sits in the boundary of $\mathcal{C}$. 
A face is said to be {\em exposed} when it is the intersection of a hyperplane (\ie a codimension-1 affine subspace of $V$) with $\mathcal{C}$.

A point $x$ in $\mathcal{C}$ is said to be {\em extreme} when there is no segment in $\partial\mathcal{C}$  containing $x$ in its interior.
The boundary $\partial \mathcal{C}$ of $\mathcal C$ is called a convex sphere. A  convex sphere is homeomorphic to a geometric sphere in a Euclidean vector space. Each face of $\mathcal{C}$ is a closed subset  of $\mathcal{C}$ (this uses the fact that the ambient vector space is finite-dimensional). The boundary of $\mathcal{C}$ is a union of faces. The dimension of a face is defined to be the dimension of its affine hull. The extreme points of $\mathcal{C}$ are its $0$-dimensional faces.

%

The set of all faces of a convex ball (or equivalently of its boundary convex sphere) is equipped with the Hausdorff topology. This is the topology induced by the Hausdorff metric on the set of all closed subsets of the convex ball (or of the convex sphere) regarded as a compact metric space.

Let $\mathcal{S}$ is a convex sphere in some finite-dimensional vector space. We record some elementary results on the topology of faces in $\mathcal{S}$ which we shall use later.

The first result follows easily  from the definition of a face.

\begin{lemma}
\label{limit of faces}
A Hausdorff limit of an arbitrary sequence of faces  in $\mathcal{S}$ is itself a face.
\end{lemma}

The next two results are simple, but need proofs.
\begin{lemma}
\label{n-dim limit}
 A  face in $\mathcal{S}$ which is a Hausdorff limit of a sequence of  $n$-dimensional faces has dimension at most $n$.
\end{lemma}
\begin{proof}
Let $(F_i)$ be a sequence of $n$-dimensional faces of $\mathcal{S}$, and suppose, seeking a contradiction, that it converges to a face $F$ of dimension greater than $n$.
Then $F$ contains $n+2$ points $p_0, \dots , p_{n+1}$ such that the vectors $\overrightarrow{p_0 p_j}\ (j=1, \dots , n+1)$ are linearly independent.
Since $F_i$ converges to $F$ in the Hausdorff topology, $F_i$ contains points $p_0(i), \dots , p_{n+1}(i)$ such that $\lim_{i\to \infty} p_j(i)=p_j$ for $j=1, \dots , n+1$.
Then the vectors $\overrightarrow{p_0(i) p_j(i)}$ are linearly independent for sufficiently large $i$, which contradicts the assumption that $F_i$ is $n$-dimensional.
\end{proof}

\begin{proposition}
\label{interior convergence}
Let $(F_i)$ be a sequence of $n$-dimensional faces of $\mathcal{S}$, and suppose that there is a sequence of points $(p_i \in F_i)$ which converges to an interior point $p$ of an $n$-dimensional face $F$ of $\mathcal{S}$.
Then $(F_i)$ converges to $F$ in the Hausdorff topology
\end{proposition}
\begin{proof}
We shall show that every subsequence of $(F_i)$ has a subsequence converging to $F$ in the Hausdorff topology, which evidently implies our proposition.
Passing to a subsequence, we can assume that $(F_i)$ converges in the Hausdorff topology, and by \cref{limit of faces}, the limit is a face, which we denote by $F'$.
By \cref{n-dim limit}, the limit $F'$ has dimension at most $n$.
On the other hand, since $p$ is the limit of the sequence $(p_i \in F_i)$, it must be contained in $F'$.
Since $p$ is an interior point, this implies that $F$ is contained in $F'$.
 By assumption, $F$ has dimension $n$. Therefore $F=F'$, which means that $F$ is a Hausdorff limit of a subsequence of $(F_i)$.
This completes the proof.
\end{proof}

\subsection{Chain-recurrent geodesic laminations}
We equip the surface $S$ with a a fixed hyperbolic metric. 
A \emph{geodesic lamination} on $S$ is a closed subset of $S$ which is a union of disjoint simple geodesics.
A geodesic lamination is said to be \emph{maximal} if it is not a proper subset of another geodesic lamination.
This is equivalent to saying that its complementary components in the surface are all ideal triangles.

Let $\lambda$   be a geodesic laminations on $S$, and $\ell$ a bi-infinite geodesic in $S$ which is disjoint from $\lambda$. We  say that $\ell$ (or  a half-leaf of  $\ell$)  \emph{spirals along} $\lambda$ if in the universal cover $\mathbb{H}^2$ of $S$ there is a lift of $\ell$ (or of the corresponding half-leaf) and a lift of a half-leaf of $\lambda$ which converge to the same point at infinity in $\mathbb{H}^2$. 
Since $\ell$ has two ends, we say that $\ell$ spirals along $\lambda$ in an end $e$ if a half leaf containing $e$ as its end spirals along $\lambda$.

A minimal component of a geodesic lamination is a sublamination which does not admit any proper sublamination.
For a geodesic lamination $\lambda$, we denote by $\lambda^c$ the union of its minimal components.
The complement $\lambda \setminus \lambda^c$ consists of a union of isolated leaves each of whose two ends spirals along some minimal component.  

The notion of chain-recurrent geodesic lamination was introduced by Thurston in \cite{ThM}.
It plays an important role in the development of Thurston's metric, in particular for what concerns geodesics and tangent vectors.
Since we do not need to use Thurston's original definition here, we adopt a necessary and sufficient condition as the definition of chain-recurrence.

\begin{definition}
 A geodesic lamination is said to be {\em chain-recurrent} if it is a Hausdorff limit of a sequence of geodesic laminations consisting of disjoint unions of simple closed geodesics.
\end{definition}
From this definition, we see that the set of chain-recurrent geodesic laminations is closed with respect to the Hausdorff topology.
Let $\chr(S)$ denote the space of chain-recurrent geodesic laminations equipped with the Hausdorff topology. We shall use the theory of train tracks originating in Thurston's notes \cite{GT3} and used in his paper \cite{ThM} for the setting of non-necessarily measured geodesic laminations.
A train track is said to be bi-recurrent if it is recurrent, \ie there is a weight system which is positive on every branch, and transversely recurrent, \ie there is a tangential measure, which we shall explain in \cref{tangential}, which takes a positive value on every branch.
Every chain-recurrent geodesic lamination is carried by a bi-recurrent train track.
(This fact is mentioned in \cite{ThM}. See \cref{bi-recurrent} for the proof.)
In this paper, we always assume that train tracks are trivalent, \ie at every switch, there are only three branches coming there, one from one direction and the other two from the other directions.

%
%

\subsection{Stretch vectors}

For every maximal chain recurrent geodesic lamination $\lambda$ on $S$ and for every $x \in \teich(S)$, there is a stretch ray $\alpha_\lambda\colon [0,\infty) \to \teich(S)$, defined in \cite{ThM}. 
Such a ray is a geodesic ray starting from $x$ (see \cite[\S4]{ThM}).

\begin{definition} Given a stretch ray $\alpha_\lambda\colon [0,\infty) \to \teich(S)$ starting at $x\in\teich(S)$, its  differential $\displaystyle\frac{d\alpha(t)}{dt}$ at $t=0$ is called the stretch vector along $\lambda$ at $x$, and is denoted by $\v_\lambda(x)$.
\end{definition}
Since $\alpha_\lambda$ is a geodesic ray starting at $x\in \teich(S)$ with respect to Thurston's asymmetric metric, $\v_\lambda(x)$ is a unit tangent vector with respect to the norm $\Vert \cdot \Vert_{\th}$ defining the Finsler structure of this metric. We denote by $\mathcal S_x$
the unit tangent sphere with respect to $\Vert \cdot \Vert_\th$ at $x \in \teich(S)$. 

The following lemma is a consequence of the continuity of ratio-maximising laminations  \cite[Theorem 8.4]{ThM} applied to the case where all the $\mu(g_i,h_i)$ (using the notation of that theorem) are maximal, which holds in the case where these laminations are associated with stretch lines.
\begin{lemma}
\label{continuous stretch}
Let $x$ be a point in $\mathcal{T}$ and, for $i=1,2,\ldots$, let $(\v_{\lambda_i}(x))$ be a sequence of stretch vectors in $\mathcal S_x$.
Then, the sequence $(\v_{\lambda_i}(x))$ converges to some vector $\v_\lambda(x)$ if and only if $(\lambda_i)$ converges to $\lambda$ in the Hausdorff topology.
\end{lemma}

Properties of a face, such as its dimension, can be deduced from the dynamical and combinatorial properties of the lamination $\lambda$.
We shall see this more precisely later in this paper.

The following two theorems  were proved in \cite{BOP}.

\sloppy
\begin{theorem}[\cite{BOP}, Theorem 4.1]
\label{face}
For any face $F$ of $\mathcal S_x$, there is a unique chain-recurrent geodesic lamination $\lambda$ such that $F$ is expressed as
\begin{equation*}
\begin{split}
&F=F_\lambda:=\{\alpha_1 \v_{\lambda_1}(x)+ \dots +\alpha_k \v_{\lambda_k}(x) \mid  \lambda_1, \dots , \lambda_k \text{ maximal chain-recurrent }  \\
&\text{geodesic laminations containing } \lambda, \  \alpha_1, \dots , \alpha_k > 0,  \alpha_1+ \dots + \alpha_k=1\}.
\end{split}
\end{equation*}
Furthermore any set of the form $F_\lambda$ for some chain-recurrent geodesic lamination $\lambda$ is a face of $\mathcal S_x$.
\end{theorem}
\begin{theorem}[\cite{BOP}, Theorem 4.3]
\label{extreme}
Any stretch vector at $x \in \teich(S)$ is an extreme point of $\mathcal S_x$.
\end{theorem}
\fussy

The following is a consequence of \cref{continuous stretch,face}.
\begin{lemma}
\label{limit face}
For each $i \in \naturals$, let $v_i$ be a tangent vector at a point $x\in S$,  lying in  some face $F_{\lambda_i}$ for some chain-recurrent geodesic lamination $\lambda_i$.
Suppose that the sequence $(v_i)$ converges to some $w \in \mathcal S_x$.
Then for any Hausdorff limit $\mu$ of a subsequence of  $(\lambda_i)$, the vector $w$ is contained in the face $F_\mu$.
\end{lemma}
\begin{proof}
Let $\mu$ be the Hausdorff limit of a subsequence of $(\lambda_i)$. By passing to a subsequence, we may assume that $\mu$ is the  Hausdorff limit of $(\lambda_i)$.
Since $v_i$ lies on $F_{\lambda_i}$, by \cref{face}, there exist maximal chain-recurrent geodesic lamination $\lambda^1_i, \dots , \lambda^k_i$ all of which contain $\lambda_i$ and positive numbers $\alpha^1_i, \dots , \alpha^k_i$ such that $v_i=\alpha^1_i \v_{\lambda^1_i}(x)+ \dots + \alpha^k_i\v_{\lambda^k_i}(x)$ with $\alpha^1_i+ \dots +\alpha^k_i=1$.
Since the dimension of a face is bounded by $6g-7$, we may assume that $k \leq 6g-7$.
Passing to a subsequence, we can further assume that $k$ is independent of $i$.

Since the set of maximal chain-recurrent geodesic laminations is closed in the Hausdorff topology, passing to a subsequence again, for each $j=1, \dots , k$, the sequence $(\lambda^j_i)$ converges to a maximal chain-recurrent geodesic lamination $\mu^j$  and $(\alpha^j_i)$ converges to $\beta^j \in [0,1]$ as $i \to \infty$.
By \cref{continuous stretch}, $(\v_{\lambda^j_i})$ converges to $\v_{\mu^j}$ as $i \to \infty$.
Since all of the $\lambda^j_i\ (j=1, \dots , k)$ contain $\lambda_i$, their limits $\mu^1, \dots , \mu^k$ all contain the Hausdorff limit $\mu$ of $(\lambda_i)$.
Therefore, the limit $w$ of $(v_i)$ lies in $F_\mu$.
\end{proof}

We can refine this result to the following proposition.

\begin{proposition}
\label{face Hausdorff}
A sequence of chain recurrent geodesic laminations $(\lambda_i)$ converges to a chain recurrent geodesic lamination $\mu$ in the Hausdorff topology of $S$ if and only if the sequence of faces $(F_{\lambda_i})$ converges to $F_\mu$ in the Hausdorff topology of the unit sphere $\mathcal{S}_s$.
\end{proposition}
\begin{proof}
We note that the Hausdorff convergence of $(F_{\lambda_i})$ to $F_\mu$ is equivalent to the conditions  that any sequence $(v_{i_j})$ taken from $F_{\lambda_{i_j}}$ for a subsequence $(F_{i_j})$ of $(F_j)$, has a  subsequence converging to a point on $F_\mu$, and that every point of $F_\mu$ is a limit of a sequence $(v_i\mid v_i \in F_{\lambda_i})$.

We first show that the \lq only  if' part.
Suppose that the sequence $(\lambda_i)$ converges to $\mu$ in the Hausdorff topology of the surface.
Then by \cref{limit face}, any limit of a subsequence of $(v_i\in F_{\lambda_i})$ lies on $F_\mu$.

We shall next show that any point on $F_\mu$ is a limit of points on $F_{\lambda_i}$.
Let $w$ be a point on $F_\mu$.
By \cref{face}, there are maximal chain recurrent geodesic laminations $\mu_1, \dots, \mu_k$ containing $\mu$  and positive numbers $\alpha_1, \dots , \alpha_k$  with $\alpha_1+ \dots + \alpha_k=1$ such that $v_\mu=\alpha_1 \v_{\mu_1}(x)+ \dots + \alpha_k \v_{\mu_k}(x)$.
For each $\mu_j\, (j=1, \dots , k)$, we choose a maximal recurrent train track $\tau_j$ carrying $\mu_j$ with its sub-train track $\tau'$ carrying $\mu$, which is independent of $j$.
Since $(\lambda_i)$ converges to $\mu$, for sufficiently large $i$, the geodesic lamination $\lambda_i$ is carried by $\tau'$.
By adding leaves carried by $\tau_j$, we can extend $\lambda_i$ to a maximal chain recurrent geodesic lamination $\lambda_i(j)$ carried by $\tau_j$ such that the sequence $(\lambda_i(j))$  converges to $\mu_j$ in the Hausdorff topology as $i\to \infty$.
Then $\alpha_1 \v_{\lambda_i(1)}(x)+ \dots + \alpha_k\v_{\alpha_k \lambda_i(k)}(x)$ is a point lying on $F_{\lambda_i}$ which converges to $w$ as $i \to \infty$.
This completes the proof of the \lq only if' part.

We now turn to the \lq if' part.
Suppose that $(F_{\lambda_i})$ converges to $F_\mu$ in the Hausdorff topology.
We shall show that any subsequence of $(\lambda_i)$ has a subsequence converging to $\mu$ in the Hausdorff topology.
We use the same symbol $(\lambda_i)$ to denote an arbitrary subsequence of $(\lambda_i)$.
Take an interior point $w$ of $F_\mu$.
Then by \cref{face}, we have an expression $w =\alpha_1 \v_{\mu_1}(x) + \dots + \alpha_k \v_{\mu_k}(x)$.
Now the largest chain recurrent geodesic lamination shared by $\mu_1, \dots, \mu_k$ is $\mu$, for otherwise $w$ lies on a subface of $F_\mu$, and cannot be an interior point.
Since $(F_{\lambda_i})$ converges to $F_\mu$ in the Hausdorff topology, there is a sequence  $(w_i \in F_{\lambda_i})$ converging to $w$.
Then, by \cref{face}, $w_i$ is expressed as $w_i=\beta^i_1\v_{\nu_1(i)}(x)+ \dots +\beta^i_{k_i}\v_{\nu_{k_i}(i)}(x)$, where $\nu_1(i), \dots , \nu_{k_i}(i)$ are maximal chain-recurrent geodesic laminations sharing $\lambda_i$ as their common chain recurrent sub-geodesic lamination, and $\beta^i_1+ \dots \beta^i_{k_i}=1$.
We can choose such expressions so that  $k_i$ is less than or equal to the dimension of the face $F_{\lambda_i}$, which is bounded by $6g-7$.
Hence, passing to a subsequence, we can assume that $k_i$ is constant, which we denote by $k$.
Since $(w_i)$ converges to $w$, and passing to a subsequence, $(\nu^i_j)$ converges to a maximal chain-recurrent geodesic lamination $\gamma_j$.
By changing the order of $\nu_1(i), \dots , \nu_k(i)$, and passing to a subsequence, we can assume that there is $l \leq k$ such that $\lim_{i \to \infty}\beta^i_j$ exists and positive for $j=1, \dots , l$ whereas $\lim_{i \to \infty}\beta^i_j=0$ if $j >l$.
Let $\lambda'_i$ be the largest common chain-recurrent sub-geodesic lamination of $\nu_1^i, \dots \nu_l^i$, and set $\beta_j=\lim_{i \to \infty} \beta^i_j$.
By definition, $\lambda_i'$ contains $\lambda_i$, and we can assume that $(\lambda_i')$ converges in the Hausdorff topology, by passing to a subsequence.
Then, since $w_i$ converges to $\beta_1 \gamma_1+ \dots \beta_l \gamma_l$ with $\beta_1, \dots , \beta_l >0$, the largest common chain-recurrent sub-geodesic lamination of $\gamma_1, \dots , \gamma_l$ contains the Hausdorff limit of $\lambda_i'$.
Since the limit of $(w_i)$ is $w$, this implies that the Hausdorff limit of   $(\lambda_i')$ is contained in  $\mu$.
It follows that the Hausdorff limit of $(\lambda_i)$ is contained in $\mu$.

We shall next show the opposite inclusion.
Let $\lambda_\infty$ be the Hausdorff limit of $(\lambda_i)$, and suppose that $\lambda_\infty$ is properly contained in $\mu$.
Then, there is a maximal chain-recurrent geodesic lamination $\lambda'_\infty$ containing $\lambda_\infty$ and intersects $\mu$ transversely.
As in the proof of the \lq only if' part, we can extend $\lambda_i$ to a maximal chain-recurrent geodesic lamination $\lambda_i'$ which converges to $\lambda'_\infty$.
Since $\lambda_i'$ contains $\lambda_i$, the stretch vector $\v_{\lambda_i'}$ is contained in $F_{\lambda_i}$.
On the other hand, since $(\v_{\lambda_i'})$ converges to $\v_{\lambda_\infty'}$, which is not contained in $F_\mu$.
Thus we have  found a sequence of points on $F_{\lambda_i}$ converging to a point outside $F_\mu$, contradicting our assumption that $(F_{\lambda_i})$ converges to $F_\mu$ in the Hausdorff topology.
Combining this with what we  proved in the preceding paragraph, we have shown that any subsequence of $(\lambda_i)$ has a subsequence converging to $\mu$.
\end{proof}
%
\subsection{Cataclysms}
The notion of cataclysm was introduced by Thurston in \cite{ThM} as a  deformation of hyperbolic structures on a surface, which generalise both earthquakes and stretch maps. Underlying the cataclysm transformation, there is a notion of cataclysm coordinates for Teichm\"uller space, which we recall now.
Let $\lambda$ be a maximal geodesic lamination and let $\MF_\lambda$ be the subspace of measured foliation space $\MF(S)$ consisting the measured foliations transverse to $\lambda$.
For a maximal chain-recurrent geodesic lamination $\lambda$ on $S$ with a fixed hyperbolic metric, we define $F(\lambda)$ to be its horocyclic  measured foliation.
Then each component of $S \setminus \lambda$ is an ideal triangle.
For every measured foliation $F$ transverse to $\lambda$, we can put a hyperbolic metric on $S$ such that the horocyclic measured foliation is isotopic to $F$.
This defines the cataclysm map $\cat_\lambda \colon \MF_\lambda \to \teich(S)$ associated with $\lambda$, which gives a coordinate system of Teichm\"{u}ller space.

\section{The space of tangential measures and the complexity of chain-recurrent laminations}
\label{tangential}
In this section, we  introduce the notion of tangential measure and study the dimension of the space of tangential measures to define the complexity for a chain-recurrent geodesic lamination.
These  are necessary for one of our main theorems, \cref{dimension}.

We start with the definition of tangential measure on a train track, which was first defined by Thurston \cite[\S 9.7]{ThL}, and explained in detail in Penner-Harer \cite[Chap.\ 1.3]{PH}.

\begin{definition}
\label{tangential measure}
Let $\tau$ be a train track.
A {\em tangential measure} $w$ on $\tau$ is a collection $\{w(b)\}$ of non-negative numbers, one  on each branch $b$ of $\tau$, which satisfies the following.
\begin{enumerate}[(1)]
\item Let $U$ be a simply-connected component of $S \setminus \tau$.
Let $b_1, \dots , b_k$ be the branches  of $\tau$ lying on $\Fr U$, which may contain the same branch twice if it is adherent to $U$ on both sides.
Then for every $j=1, \dots, k$, we have $\displaystyle w(b_j) \leq \sum_{i \neq j} w(b_i)$.
\item Let $U$ be a component of $S \setminus \tau$ homeomorphic to an open annulus, consisting of a simple closed branch $c$ and other branches $b_1, \dots , b_k$ (which may contain the same branch twice in the same way as above.
Then $\displaystyle w(c)\leq \sum_{j=1}^k w(b_j)$.
\end{enumerate}
Following \cite{PH}, we say that a train track $\tau$ is {\em bi-recurrent} when it is recurrent and carries a tangential measure whose values are all positive.
\end{definition}

We introduce the following equivalence relation between tangential measures.
Let $s$ be a switch of a train track $\tau$ to which two branches $b_1, b_2$ come from one direction and another branch $b_3$ from the other direction.
Suppose that $w$ is a tangential measure which takes values $v_1, v_2, v_3$ on the three branches $b_1, b_2, b_3$ respectively.
We define  two kinds of elementary moves of $w$ at $s$ to be the change of $(v_1, v_2, v_3)$ to $(v_1+c, v_2+c, v_3-c)$ for a positive number $c$ less than $v_3$, and the change of $(v_1, v_2, v_3)$ to $(v_1-d, v_2-d, v_3+d)$ for a positive number $d$ less than $\min\{v_1, v_2\}$.
We define two tangential measures to be equivalent when one can be obtained from the other by a finite sequence of elementary moves.
The equivalence class of a tangential measure $w$ is denoted by $[w]$.

The space of positive tangential measures on $\tau$ modulo this equivalence relation is denoted by $M^*(\tau)$.
It is the quotient space of the space of tangential measures whose kernel is linearly generated by $2v_1+v_2+v_3$  at $s$, where $s$ ranges over all switches of $\tau$.
Therefore $M^*(\tau)$ is the positive part of a vector space, and we can define its dimension to be the dimension of this vector space. 
We note that except for the case when $M^*(\tau)$ is trivial,  we can construct a basis totally contained in $M^*(\tau)$ from any basis of the vector space by adding a scalar multiple of a positive element, and any element can be expressed as a non-negative linear combination of elements contained in the basis.
Therefore, in particular, when $\tau$ is bi-recurrent, the dimension of $M^*(\tau)$ is equal to its topological dimension.

%

Until now, we have argued fixing a train track $\tau$.
We now turn to considering a finer train track $\tau'$ carried by $\tau$, and study the relation between $M^*(\tau')$ and $M^*(\tau)$.
For simplicity, we assume that all switches of $\tau'$ are also switches of $\tau$.
This condition can be realised by performing shifting operations on $\tau'$ evidently.
We can regard $\tau'$ as being immersed on $\tau$, and hence each branch of $\tau'$ as an ordered sequence of consecutive branches of $\tau$.

Now, let $w$ be a tangential measure on $\tau$.
For each branch $b'$ of $\tau'$, let $b_1, \dots, b_k$ be the branches (the same branch may appear more than once) which $b'$ passes in this order.
We note that since we assumed that every switch of $\tau'$ lies at that of $\tau$, if $b'$ passes a branch of $\tau$, then  it passes from one endpoint to the other entirely.
Then, we define a tangential measure $w'$ on $\tau'$ by setting $w'(b')$ to be $w(b_1)+ \dots + w(b_k)$.
It is easy to check that $w'$ is a tangential measure on $w'$.
We consider its equivalence class $[w']$ in $M^*(\tau')$.

\begin{lemma}
The equivalence class $[w']$ depends only on the equivalence class $[w]$ of $w$, and hence there is a linear map $\phi_{\tau, \tau'} \colon M^*(\tau) \to M^*(\tau')$ taking $[w]$ to $[w']$.
\end{lemma}
\begin{proof}
Let $s$ be a switch of $\tau$, where two branches $b_1, b_2$ come from one direction and another one $b_3$ comes from the other direction.
If the interior of a branch $b'$ of $\tau'$ passes through $s$, then  it passes through either $b_1$ and $b_3$ or $b_2$ and $b_3$ each time it passes through $s$.
Since $w(b_1)$ and $w(b_3)$ are changed to $w(b_1)+c$ and $w(b_3)-c$, their sum is preserved.
The same holds for $w(b_2)$ and $w(b_3)$.
Thus we see that $w'(b')$ does not change in this case.

Next suppose that $b'$ has $s$ as one of its endpoints.
Then starting from $s$, the branch $b'$ proceeds in one of $b_1, b_2$ and $b_3$.
Suppose that it proceeds in $b_3$.
Then there are branches $b'_1, b'_2$ adjacent to $b'$ which proceed in $b_1$ and $b_2$ respectively.
By the elementary move changing $(w(b_1), w(b_2), w(b_3))$ to $(w(b_1)+c, w(b_2)+c, w(b_3)-c)$ with $c$ positive or negative, $(w'(b_1), w'(b_2), w')$ is also changed to $(w'(b_1')+c, w'(b_2')+c, w'(b_3)-c)$, and hence the equivalence class remains unchanged.
We can argue in the same way also when $b'$ proceeds to $b_1$ or $b_2$.
Thus we have shown that the equivalence class of $w'$ does not depend the choice of a representative of the class $[w]$, and hence $\phi_{\tau, \tau'}$ is well defined.
The linearity of the map $\phi_{\tau, \tau'}$ is obvious.
\end{proof}

For a chain-recurrent geodesic lamination $\lambda$, we consider a  train track carrying $\lambda$ such that each component $U$ of $S \setminus \tau$ is isotopic to  a deformation retract of the corresponding component $U'$ of $S\setminus \lambda$ in such a way each switch lying in $U$ corresponds to an ideal vertex of $U'$.
We call such a train track {\em adapted} to $\lambda$.
By definition, every branch of $\tau$ carries a leaf of $\lambda$, and hence $\tau$ is recurrent by the chain-recurrency of $\lambda$.
%


We shall next define tangential measures on a train track $\tau$ which measured foliations transverse to  $\tau$ induce.
This will be used in the next section essentially.
Let $\MF(S)$ be the space of isotopy classes of measured foliations on $S$ (including the empty set which corresponds to the origin), which is homeomorphic to the Euclidean space of dimension $6g-6$.
\begin{definition}
A measured foliation $F$ on $S$ is said to be transverse to a train track $\tau$ when $F$ is isotoped so that every branch of $\tau$ is transverse to the leaves of $F$ and contains no singularities of $F$.
Let $\MF_\tau$ be the subset of $\MF(S)$ consisting of the isotopy classes of all measured foliations transverse to $\tau$.
\end{definition}

For a measured foliation $F$ transverse to $\tau$, we let $w_F$  be a  function on the set of branches of $\tau$ defined by $w_F(b)=m_F(b)$, where $m_F$ denotes the transverse measure of $F$.
We can easily verify that $w_F$ is a tangential measure, and that isotopic measured foliations give the same equivalence class of tangential measure.
Therefore, we can define a map $w_\tau \colon \MF_\tau \to M^*(\tau)$.

\begin{lemma}
\label{surjection}
The map $w_\tau$ is a linear surjection.
\end{lemma}
\begin{proof}
The linearity is evident.
We shall prove the surjectivity.
Let $w$ be a tangential measure on $\tau$.
Take a tied neighbourhood $N_\tau$ of $\tau$.
Recall that $N_\tau$ consists of tied rectangles each of which is a product of a branch and an interval.
Let $R_b$ be the rectangle corresponding to a branch $b$.
We give the ties of $R_b$ a transverse measure whose total measure is equal to $w(b)$.
This gives a measured foliation on $N_\tau$ each of whose leaves is either a tie or a union of two ties  corresponding to a switch.
We extend this measured foliation to the entire surface $S$ as a measured foliation.
For a simply-connected component and an annulus with one $C^1$-boundary component of $S \setminus N_\tau$, the conditions (1) and (2) in \cref{tangential measure} guarantee that this is possible.
For other types of components of $S \setminus N_\tau$, it is easy to extend it, in particular for a non-annular component,  by choosing a spine and make it a singular leaf.
(See \S5 Fathi--Laudenbach--Poénaru \cite{FLP}.)
By construction the $w_\tau$-image of the isotopy class of this foliation is $[w]$.
\end{proof}

The following is an immediate consequence of our definitions of the two maps.
\begin{lemma}
Suppose that $\tau'$ is carried by $\tau$.
Then we have $\phi_{\tau, \tau'}\circ w_\tau= w_{\tau'}|_{\MF_{\tau}}$. 
\end{lemma}

%

The weight systems on $\tau$, which are regarded as transverse measures,  constitute the non-negative part of a vector space, which we denote by $M(\tau)$.
\begin{lemma}
\label{dim tangential}
Let $\tau$ be a  train track adapted to a chain-recurrent geodesic lamination $\lambda$.
The dimension of $M^*(\tau)$ is equal to the dimension of $M(\tau)$.
\end{lemma}
\begin{proof}
Let $\Sigma$ be the minimal supporting surface of $\lambda$.
For any tangential measure $w$ on $\tau$,  by the construction of the proof of \cref{surjection}, there is a measured foliation $F$ transverse to $\Fr \Sigma$ such that $w_\tau(F)=w$.
Following \cite{PH}, for a weight system $v$ on $\tau$, we can consider the intersection number $i(v,w)$ between $v$ and $w$, defined by $i(v,w)=\sum_{b} v(b) w(b)$, where $b$ ranges over the branches of $\tau$.
This defines a bi-linear function $i \colon M(\tau) \times M^*(\tau) \to \reals_+$.
Measured foliations on $\Sigma$ transverse to the boundary are uniquely determined by the measure of pants decomposition and the transversals.
Since such curves can be taken to be carried by $\tau$, we see that the intersection number is non-degenerate with respect to $M^*(\tau)$.
In the same way, the measured laminations carried by $\tau$ are uniquely determined by the measure of finitely many arcs transverse to $\Fr \Sigma$ which can be regarded as measured foliations transverse to $\tau$.
Therefore, the intersection number is also non-degenerate with respect to the space of weight systems on $\tau$.
This shows that their dimensions are the same.
\end{proof}

\begin{lemma}
\label{same dimension}
Suppose that two train tracks $\tau_1$ and $\tau_2$ are both adapted to $\lambda$.
Then $\dim M^*(\tau_1)=\dim M^*(\tau_2)$.
\end{lemma}
\begin{proof}
Since both $\tau_1$ and $\tau_2$ are adapted to $\lambda$, there is a train track $\tau_3$ adapted to $\lambda$ which is carried by both $\tau_1$ and $\tau_2$.
We consider the map $\tau_{\tau_1, \tau_3} \colon M^*(\tau_1) \to M^*(\tau_3)$.

Since $\tau_1$  is adapted to $\lambda$, we can obtain $\tau_3$ from $\tau_1$  by shifting and splitting without \lq collision type' splitting.
Then $M(\tau_3)$ has the same dimension as  $M(\tau_1)$.
Therefore by \cref{dim tangential}, we see that $\dim M^*(\tau_1)=\dim M^*(\tau_3)$.
The same holds for $\tau_2$, and we have $\dim M^*(\tau_1)=\dim M^*(\tau_2)$.
\end{proof}

For a train track adapted to a chain-recurrent geodesic lamination this lemma combined with \cref{surjection} implies the following.

\begin{corollary}
\label{bi-recurrent}
A train track adapted to a chain-recurrent geodesic lamination is bi-recurrent.
\end{corollary}
\begin{proof}
Let $\lambda$ be a chain-recurrent geodesic lamination, and suppose that $\tau$ is a train track adapted to $\lambda$.
Take any measured foliation $F$ transverse to $\lambda$.
Then if we take a sufficiently fine train track $\tau$ adapted to $\lambda$, we see that $F$ lies in $\MF_\tau$, and hence $M^*(\tau)$ has a non-zero image $w_\tau(F)$.
In particular, $\dim M^*(\tau)>0$.
By \cref{same dimension}, this holds for every train track adapted to $\lambda$, and hence it has a non-trivial tangential measure.
\end{proof}

Now we define the complexity of  a chain-recurrent geodesic lamination.
\begin{definition}
\label{complexity}
Let $\lambda$ be a chain-recurrent geodesic lamination.
We define the complexity of $\lambda$ to be the dimension of $M^*(\tau)$ for any train track $\tau$ adapted to $\lambda$, which was shown to be independent of the choice of $\tau$ by \cref{same dimension}, and denote it by $C(\lambda)$.
\end{definition}

The complexity which we  defined in \cref{complexity} turns out to be upper semi-continuous with respect to the Hausdorff topology.
\begin{lemma}
\label{complexity semi-continuous}
Let $(\lambda_i)$ be a sequence of  chain-recurrent geodesic laminations which converges to a chain-recurrent geodesic lamination $\mu$ in the Hausdorff topology.
Then $C(\mu)\geq \limsup_{i \to \infty} C(\lambda_i)$.
\end{lemma}
\begin{proof}
Let $\sigma$ be a  train track adapted to $\mu$.
Since $(\lambda_i)$ converges to $\mu$ in the Hausdorff topology, we can choose a train track $\tau_i$ adapted to $\lambda_i$ so that $\tau_i$ is carried by $\sigma$ for sufficiently large $i$.
Then $\tau_i$ is a sub-track of a train track $\sigma_i$ adapted to $\mu$ which is also carried by $\sigma$.
By \cref{same dimension}, $\dim M^*(\sigma_i)=\dim M^*(\sigma)$, whereas by \cref{dim tangential} $\dim M^*(\tau_i)=\dim M(\tau_i) \leq \dim M(\sigma_i)= \dim M^*(\sigma_i)$, and we are done.
\end{proof}

%

\section{Dimensions of faces}
In  \cite[Theorem 10.1]{ThM}, Thurston showed that for each $x\in \mathcal{T}(S)$, the complement of the \lq facets' in $\mathcal{S}_x$, that is, the faces of maximal dimension,  which correspond (using our notation in \cref{face}) to the $F_\lambda$  with $\lambda$ being a simple closed geodesic, has positive Hausdorff codimension.
We shall show that an argument similar to his proof of this theorem gives a description of the dimension of a face $F_\lambda$ in terms of $C(\lambda)$ which we defined in the preceding section.

\begin{theorem}
\label{dimension}
Let $F_\lambda$ be a face of $\mathcal S_x$ associated with a chain-recurrent geodesic lamination $\lambda$ as in \cref{face}.
Then  we have
$$\dim(F_\lambda)=6g-6-C(\lambda).$$
\end{theorem}
\begin{proof}
In the same way as in the argument in \S10 of \cite{ThM}, we shall use the cataclysm coordinates for the proof.
As shown in \cite[p.\ 50]{ThM}, the chain-recurrent geodesic lamination $\lambda$ can be extended to a maximal chain recurrent geodesic lamination $\mu$.

%
%
Thurston's cataclysm coordinates with respect to $\mu$ give a diffeomorphism $c_\mu \colon \MF_\mu \to \teich(S)$, where $\MF_\mu$ denotes the subspace of $\MF(S)$ consisting of all measured foliations transverse to $\mu$ up to isotopy.
Choose a point $x$ in $\mathcal{T}(S)$, and let $G$ be the measured foliation with $c_\mu(G)=x$.

Take a train track $\sigma$ adapted to $\mu$ which is transverse to $G$.
Since $\lambda$ is a sublamination of $\mu$, there is a train track $\tau$ adapted to $\lambda$ which is a sub-track of $\sigma$.
Since $\MF_\sigma$ is a subspace of $\MF_\mu$, restricting $c_\mu$ to $\MF_\sigma$, we get a map $c_\sigma \colon \MF_\sigma \to \teich(S)$.
Let $dc_\sigma\colon \MF(S) \to T_x \teich(S)$ be its differential at $G$, where we have identified the tangent space of $\MF_\sigma$ with its ambient space $\MF(S)$.

Since $\tau$ is a sub-track of $\sigma$, any foliation in $\MF_\sigma$ is also contained in $\MF_\tau$, and hence we can  consider the map $\mathbf w \colon \MF_\sigma  \to M^*(\tau)$ taking $F \in \MF_\sigma$ to  $w_\tau(F)\in M^*(\tau)$ defined in the preceding section, regarding $F$ as a foliation in $\MF_\tau$.
Let  $d\mathbf w \colon \MF_\sigma \to M^*(\tau)$ be the differential of  $\mathbf w$ at  $G$ restricted to $\MF_\sigma$, identifying the positive part of the tangent space of $M^*(\tau)$  with $M^*(\tau)$.

On the other hand, let $\mathcal F_\lambda$ be the subset of $\MF_\sigma$ consisting of all measured foliations whose images under $d\mathbf w$ are scalar multiples of $d \mathbf w(G)$, \ie $\mathcal F_\lambda=\{F \in \MF_\sigma \mid d\mathbf {\mathbf w}(F) \in \reals_+ d\mathbf w(G)\}$.
We note that this coincides with the space of  tangent vectors which stretch $\lambda$ most.
%
%
%
%
%
%
%
Since the entire space $\MF_\sigma$ has dimension equal to $\dim \teich(S)$,  the subset $F_\lambda$ has  dimension $\dim\teich(S) -(\dim(M^*(\tau)-1)=6g-5-C(\lambda)$, and hence its intersection with the unit sphere in $T_x\teich(S)$ has  dimension equal to $6g-6-C(\lambda)$.
Since $F_\lambda $ consists of vectors infinitesimally stretching $\lambda$ most, $(d\mathbf w)\circ (dc_\sigma)(F_\lambda)$ is contained in $\mathcal F_\lambda$.

Let $v$ be an interior point of the face $F_\lambda$.
Let $(u_i)$ be a sequence on $\mathcal S_x$ converging to $v$, none of whose elements contained in $F_\lambda$.
Then $u_i$ is contained in a face $F_{\lambda_i}$ such that $(\lambda_i)$ converges to $\lambda$ in the Hausdorff topology by \cref{face Hausdorff}.
This implies that $\lambda_i$ is carried by $\tau$ for sufficiently large $i$.
Since $\lambda_i$ intersects $\lambda$ transversely, $u_i$ cannot stretch $\lambda$ most among the chain-recurrent geodesic laminations carried by $\tau$, and hence $(d\mathbf w)\circ (d c_\sigma)(u_i)$ does not lie in $\mathcal F_\lambda$.
It follows that the dimension of $F_\lambda$ in a small neighbourhood of  $v$ is the same as the dimension of $\mathcal F_\lambda$.
Thus, we have shown that  $F_\lambda$ has dimension equal to $6g-6-C(\lambda)$.
\end{proof}

\section{Invariance of combinatorial structure}
We know from \cref{face} that the set of faces of the unit sphere $\mathcal S_x$ in the tangent space at any point $x$ of $\teich(S)$ is in natural one-to-one correspondence with the set of chain-recurrent geodesic laminations on the surface. 
In this section, we compare unit spheres based on different points of $\teich(S)$.
To distinguish faces in $\mathcal S_x$ and those in $\mathcal S_y$, we use symbols such as $F_\lambda(x) \in \mathcal S_x$ and $F_\lambda(y) \in \mathcal S_y$ to denote the faces corresponding to $\lambda$ at $x$ and $y$.

We say that a map from $f \colon \mathcal S_x \to \mathcal S_y$ is a combinatorial isomorphism if (1) $f$ is a homeomorphism and (2) it takes a face of $\mathcal S_x$ to a face of $\mathcal S_y$.
We note that we do not require the map $f$ to be linear on faces.
\begin{theorem}
\label{combinatorially invariant}
For any $x, y \in \teich(S)$, there is a combinatorial isomorphism $f\colon \mathcal{S}_x \to \mathcal{S}_y$ taking $F_\lambda(x)$ to $F_\lambda(y)$.
\end{theorem}
\begin{proof}
We construct $f$ step by step starting from points corresponding to stretch vectors.
Let $V_x$ be the set of stretch vectors on $\mathcal S_x$ and $V_y$ the one on $\mathcal S_y$.
For each stretch vector $\v_\lambda(x) \in V_x$, we define $f(\v_\lambda(x))$ to be $\v_\lambda(y) \in V_y$.
The map is obviously bijective, and by \cref{continuous stretch} it is continuous in both directions.
Since every $0$-dimensional face consists of one stretch vector,  this gives a homeomorphism between the $0$-skeletons of $\mathcal S_x$ and $\mathcal S_y$.

We next extend $f$ to the $1$-dimensional faces.
By \cref{extreme,dimension}, for each $1$-dimensional face $F$, there are stretch vectors $\v_{\lambda_1}, \v_{\lambda_2}$ which span $F$ in the sense that every $v \in F$ is expressed as $t\v_{\lambda_1}(x)+(1-t) \v_{\lambda_2}(x)$, and $\lambda_1$ and $\lambda_2$ share a chain-recurrent geodesic lamination whose complexity is $6g-7$.
We extend $f$ by setting $f(v)=t\v_{\lambda_1}(y)+(1-t) \v_{\lambda_2}(y)$.

We shall show that the map is continuous on the union of one-dimensional faces on $\mathcal S_x$.
Let $(F_{\lambda_i}(x))$ be a sequence of one-dimensional faces on $\mathcal S_x$, and 
suppose that $(v_i \in F_{\lambda_i}(x))$ converges to $w \in \mathcal S_x$.
Since $(\lambda_i)$ is a sequence of chain-recurrent laminations, passing to a subsequence, it converges to a chain-recurrent lamination $\mu$ in the Hausdorff topology.
Since $v_i$ is contained in $F_{\lambda_i}$, by \cref{limit face}, the limit $w$, which is also the limit of any subsequence of $(v_i)$,  is contained in $F_\mu$.
Furthermore, since $C(\lambda_i)=6g-7$ by \cref{dimension}, we see  by \cref{complexity semi-continuous} that $C(\mu)\geq 6g-7$.

By \cref{extreme}, the endpoints of $F_{\lambda_i}$ correspond to stretch vectors $\v_{\lambda_i^+}(x)$ and $\v_{\lambda_i^-}(x)$ along maximal chain-recurrent geodesic laminations $\lambda_i^+$ and $\lambda_i^-$ both containing $\lambda_i$.
By \cref{face}, for each $i$, there is $t_i \in [0,1]$ such that $v_i=t_i\v_{\lambda_i^+}(x) + (1-t_i)\v_{\lambda_i^-}(x)$.
Passing to a subsequence so that all of $(\lambda_i^+), (\lambda_i^-)$ and $(t_i)$ converge to $\mu^+, \mu^-$ and $t$ respectively, we have $w=t\v_{\mu^+}(x)+(1-t)\v_{\mu^-}(x)$, and both $\mu^+$ and $\mu^-$ are maximal and contain $\mu$.
(Note that it is possible that $\mu^+=\mu=\mu^-$.)
Since, by definition, we have $f(t_i\v_{\lambda_i^+}(x)+(1-t_i)\v_{\lambda_i^-}(x))=t_i\v_{\lambda_i^+}(y)+(1-t_i)\v_{\lambda^-_i}(y)$ and $f(w)=t\v_{\mu_+}(y)+(1-t)\v_{\mu_-}(y)$, we see that $(f(v_i))$ converges to $f(w)$.
Thus $f$ can be extended   as a homeomorphim from the $1$-skeleton of $\mathcal S_x$  to the $1$-skeleton of $\mathcal S_y$ by taking $F_\lambda(x)$ to $F_\lambda(y)$ linearly for each $\lambda$ with $C(\lambda)\geq 6g-7$.

From dimension $2$ on, we abandon the linearity, and extend $f$ to a homeomorphism inductively.
First, we shall choose \lq centres' for faces of $\mathcal S_x$ and $\mathcal S_y$.
In the case of a simplicial complex, the barycentres play this role.
Since our convex spheres have more complicated combinatorial structures, we need to choose centres rather arbitrarily in the interiors of faces in such a way that limits of centres are always centres.
This can be done as follows.

For a face $F$ of dimension $n$ in $\mathcal S_x$ or $\mathcal S_y$, towards which other faces accumulate in the Hausdorff topology, we fix some interior point $w \in F$, and consider a codimension-$n$ hyperplane $H$ (in $T_x \teich(S)$ or $T_y \teich(S)$) transversely intersecting $F$ at $w$.
For any face $G$ sufficiently near to $F$ in the Hausdorff topology, $G$ is transverse to $H$ and in particular $\Int G$ intersects $H$.
Then we can choose a centre $v$ on $G$ which converge to $w$ in such a way that if there is a continuous family of faces converging to $H$, then  the corresponding centres constitute a continuous curve converging to $w$ on $H$.

For $n \geq 2$, suppose that $f$ is extended to a homeomorphism between the $(n-1)$-skeletons taking $F_\lambda(x)$ to $F_\lambda(y)$ for each $\lambda$ with $C(\lambda) \geq 6g-6-(n-1)$.
Let $F_\lambda(x)$ be an $n$-dimensional face of $\mathcal S_x$ corresponding to $\lambda$.
Then $C(\lambda)=6g-6-n$ by \cref{dimension}, and hence $F_\lambda(y)$ is also an $n$-dimensional face.
Since $f$ is already defined on  the boundary $F_\lambda(x)$, which is homeomorphic to the $(n-1)$-dimensional sphere, as a homeomorphism, we can extend $f$ to a homeomorphism from $F_\lambda(x)$ to $F_\lambda(y)$ which takes the centre of $F_\lambda(x)$ to that of $F_\lambda(y)$, using the rays issuing at the centres.
Thus we have extended $f$ to the $n$-skeleton of $\mathcal S_x$, taking $F_x(\lambda)$ to $F_y(\lambda)$ for $\lambda$ with $C(\lambda)=6g-6-n$.

Let us prove the continuity of $f$ on the $n$-skeleton.
Consider a converging sequence $(v_i) \in \Int F_{\lambda_i}(x)$ such that $F_{\lambda_i}(x)$ is contained in the $n$-skeleton.
Then $C(\lambda_i) \geq 6g-6-n$.
Passing to a subsequence, $(\lambda_i)$ converges in the Hausdorff topology to a chain-recurrent geodesic lamination $\mu$, and we have, by \cref{complexity semi-continuous}, $C(\mu) \geq 6g-6-n$.
By \cref{face Hausdorff}, $F_{\lambda_i}(x)$ converges to $F_\mu(x)$ in the Hausdorff topology, and hence  $(v_i)$ converges to a point $w$ in $F_\mu(x)$.
Let $c_i(x)$ be the centre which we chose above for $F_{\lambda_i}(x)$ and $c(x)$ that of $F_\mu(x)$.
In the same way, in $\mathcal S_y$, we let $c_i(y)$ and $c(y)$ be centres of $F_{\lambda_i}(y)$ and $F_\mu(y)$ respectively.

For a point $v$ on a face $F$ with centre $c$, we denote by $b(v,c)$ the point on $\partial F$ on which the ray issued from $c$ to $v$ land, assuming that $v$ does not coincide with $c$.
Then $v$ can be expressed as $v=tc+(1-t)b(v,c)$ for some $t \in [0,1)$.
In the above construction, we defined $f$  in such a way  that $f(c_i(x))=c_i(y)$, $f(c(x))=c(y)$, $f(v_i)=f(t_ic_i(x)+(1-t_i) b(v_i, c_i(x)))=t_ic_i(y)+(1-t_i)f(b(v_i, c_i(x)))$, and $f(w)=f(tc(x)+(1-t)b(w,c))=tc(y)+(1-t)f(b(w,c))$.
Since $(\lambda_i)$ converges to $\mu$, we see that $(c_i(x))$ converges $c(x)$ whereas $(c_i(y))$ converges to $c(y)$.
Since $(v_i)$ converges to $w$, we also see that $(t_i)$ converges to $t$.
Thus we have $\lim_{i \to \infty} f(v_i)=f(w)$, and the continuity of $f$ holds.
Since the same is true by exchanging the roles of $x$ and $y$, we a see that $f^{-1}$ is also continuous, and hence $f$ is a homeomorphism.
By construction, $f$ takes an $n$-dimensional face of $\mathcal S_x$ to an $n$-dimensional face of $\mathcal S_y$.
Thus we have proved that $f$ is a combinatorial isomorphism.
\end{proof}

\section{Combinatorial automorphisms}
Let $h \colon S \to S$ be a diffeomorphism.
Then $h$ naturally induces a homeomorphism $h_* \colon \chr(S) \to \chr(S)$.
By \cref{face}, this gives rise to a permutation on  the set of faces of $\mathcal S_x$, which we denote also by $h_*$.
We shall now show that this permutation induces a combinatorial automorphism of $\mathcal S_x$, and that except in the case when the surface $S$ has genus 2, this gives a natural isomorphism between the extended mapping class group of the surface $S$ and the group of combinatorial automorphisms of $\mathcal S_x$.

\begin{lemma}
\label{induced auto}
The permutation of the faces of $\mathcal S_x$ given by  $h_* \colon \chr(S) \to \chr(S)$ induces a combinatorial automorphism of $\mathcal S_x$.
\end{lemma}
\begin{proof}
The proof is the same as that of \cref{combinatorially invariant}. We can choose centres for faces in a consistent way, and obtain a homeomorphism realising $h_*$. 
\end{proof}

Using this lemma, we prove that for any combinatorial automorphism of $\mathcal S_x$ there is an auto-diffeomorphism of $S$ inducing the same permutation on the set of faces.
We note that by \cref{combinatorially invariant}  the combinatorial type of $\mathcal S_x$ does not depend on the basepoint $x$. More precisely, we have the following.

We denote by $\Gamma^*(S)$ the extended mapping class group of $S$, that is, the group of homotopy classes of  homeomorphisms of $S$.
For any point $x$ in $S$, we denote by $\mathrm{Aut}(\mathcal{S}_x)$ the group of combinatorial automorphisms of $\mathcal{S}_x$.

\begin{theorem}
\label{combinatorial auto}
For each $x\in \mathcal{T}$, there is an epimorphism
$h_x:  \Gamma^*(S)\to \mathrm{Aut}(\mathcal{S}_x)
$ such that any element of $\Gamma^*(S)$ and its image by $h_x$ induce the same permutation of the set of faces of $\mathcal{S}_x$. Furthermore, each such epimorphism $h_x$ is an isomorphism except in the case when the genus of the surface $S$ is 2, in which case $h_x$  has kernel isomorphic to $\mathbb{Z}_2$ generated by the hyperelliptic involution of $S$.

\end{theorem}
\begin{proof} 
By \cref{dimension}, there is a one-to-one correspondence between the set of  simple closed geodesics of $S$ and the set of $(6g-7)$-dimensional faces of $\mathcal S_x$.
Two simple closed geodesics $a, b$ are disjoint if and only if the corresponding faces $F_a$ and $F_b$ intersect.
Therefore, any combinatorial automorphism $f$ of $ \mathcal S_x$ induces a simplicial isomorphism of the curve complex of $S$. Ivanov's theorem \cite{Iv} says that such a simplicial isomorphism is induced by a diffeomorphism $h: S\to S$ and that the resulting epimorphism
from $ \Gamma(S)$ to  the group of simplicial automorphisms of the complex of curves 
 is an isomorphism except in the case of a surface of genus 2, in which case it has kernel is the group of order two generated by the hyperelliptic involution. 
Now the faces of $\mathcal S_x$ correspond one-to-one to the set of chain-recurrent geodesic laminations $\chr(S)$, and by \cref{face Hausdorff}, Hausdorff convergence in $\chr(S)$ corresponds to  Hausdorff convergence of the faces.
Since every chain-recurrent geodesic lamination is a Hausdorff limit of simple closed geodesics, $f$ and $h$, which induce the same permutation of the set of simple closed geodesics, induce the same homeomorphism of $\chr(S)$.
This in turn shows that the combinatorial automorphism induced by $h_*$ coincides with $f$ as a permutation on the set of faces of $\mathcal S_x$.
\end{proof}

\section{Cotangent space and codimension}
In this last section, we shall give an answer to a part of \cite[Question 7.12]{HOP}. The question is the following: what are the dimensions, codimensions, face-dimensions
and adherence-dimensions of arbitrary faces of the unit spheres in the tangent or cotangent spaces at some point of Teichm\"uller space equipped with  Thurston's  Finsler structure? 
Our \cref{dimension} gives an answer for the dimensions of the faces in the unit tangent spheres.
In the present section we solve the problem of codimension for the faces in unit tangent spheres and the exposed faces in the unit cotangent spheres. We also give a necessary and sufficient condition for a face of a unit cotangent sphere to be exposed.

We first turn to the unit spheres of the cotangent spaces.
Let $\ell \colon \ml(S) \times \teich(S) \to \reals_+$ be the length function, which takes a pair $(\lambda, x)$ to the length of the measured lamination $\lambda$ with respect to the hyperbolic structure $x$.
We can consider its differential with respect to the second coordinate, obtaining  $d\ell_x \colon \ml(S) \to T^*_x \teich(S)$.
Considering $\log \ell$ instead of $\ell$, the differential descends to $d \log \ell_x \colon \pml(S) \to T^*_x \teich(S)$.
Thurston proved in \cite{ThM} that the map $d\log\ell_x$ is an embedding of $\pml(S)$ whose image is  the unit cotangent sphere with respect to the norm $\Vert \cdot \Vert_\th^*$ dual to $\Vert \cdot \Vert_\th$, which we denote by $\mathcal S^*_x$.

In \cite{HOP}, the cotangent spaces $T^*\teich(S)$ were used to prove the infinitesimal mapping class group rigidity of Thurston's metric.
It was also proved there that for any point $x$ in Teichm\"uller space and for any  projective measured lamination $[\lambda]$, there is an associated exposed face in the cotangent space at $x$ defined by $\{d\log\ell_x([\mu])\mid |\mu| \subset |\lambda|\}$.
We denote this exposed face by $F^*_{|\lambda|}$.
In the same paper,  the notion of codimension of faces was introduced:

\begin{definition}
\label{codimension}
Given a face $F^*$ of $\mathcal S_x^*$, we define its codimension, denoted by $\codim F$, to be the dimension of the set of vectors $v$ in $\mathcal S_x$ such that $w^*(v)=1$ for all $w^* \in F^*$.
In the same way, we define the codimension of a face of $\mathcal S_x$ by exchanging the roles of vectors and covectors.
We say that $v$ is {\em orthogonal} to $F^*$ if $w^*(v)=1$ for all $w^* \in F$.
We also use this term for covectors in the same way.
\end{definition}

We now give a description of the codimension of a face corresponding to the support of a projective measured lamination.

\begin{theorem}
\label{codimension of face}
Let $[\lambda]$ be a projective lamination on $S$. 
Then $v \in \mathcal S_x$ is orthogonal to $F^*_{|\lambda|}$ if and only if $v$ is contained in the face $F_{|\lambda|}$ of $\mathcal S_x$.
In particular, we have $\codim(F^*_{|\lambda|})=6g-6-C(|\lambda|)$.
\end{theorem}
\begin{proof}
Once we show the first statement, the second statement directly follows from \cref{dimension}.
We shall first prove the \lq if' part of the first statement.
Suppose that $v$ is contained in $F_{|\lambda|}$.
Then $v$ infinitesimally stretches $|\lambda|$ most, and hence it does so on any measured lamination $\mu$ whose support is contained in $|\lambda|$.
Therefore,  $v(\ell_x(\mu))=\ell_x(\mu)$  for any $\mu$ with $|\mu|\subset|\lambda|$, which implies that $d\log\ell_x(\mu)(v)=1$, and hence $v$ is orthogonal to $F^*_{|\lambda|}$.

We shall next show the \lq only if' part by proving its contraposition.
Suppose that $v$ is not contained in $F_{|\lambda|}$.
By \cref{face}, it is contained in the interior of some face $F_\nu$ for a chain recurrent geodesic lamination $\nu$ not containing $|\lambda|$, and $v$ is expressed as $v=\alpha_1 \v_{\nu_1}+ \dots + \alpha_k \v_{\nu_k}$ with $\alpha_1, \dots, \alpha_k >0$ and $\alpha_1+ \dots + \alpha_k=1$ such that $\nu$ is the largest common chain recurrent geodesic lamination of the maximal chain-recurrent geodesic laminations $\nu_1, \dots , \nu_k$.
Then one of $\nu_1, \dots , \nu_k$, say $\nu_1$,  intersects $|\lambda|$ transversely.
Then we have $\v_{\nu_1}(\lambda)< \ell_x(\lambda)$ by \cite[Lemma 2.7]{BOP}, which implies that $v(\lambda) < \ell_x(\lambda)$.
Therefore, $v$ is not orthogonal to $F^*_{|\lambda|}$.
\end{proof}

As remarked in \cite{HOP}, there are faces which do not have the form $F_{|\lambda|}$.
Still we can show the following.

\begin{theorem}
\label{boundary face}
Let $F^*$ be a face of $\mathcal S^*_x$.
Then there is a projective lamination $[\lambda]$ such that either $F^*=F^*_{|\lambda|}$, or $F^*$ is a proper subface of $F^*_{|\lambda|}$ and the interior of $F^*$ consists of the images under $d\log\ell_x$ of projective measured laminations whose supports are equal to $|\lambda|$.
\end{theorem}
\begin{proof}
Let $F^*$ be a face of $\mathcal S^*_x$, and let $P$ be an interior point of $F^*$.
Let $[\lambda]$ be the projective measured lamination such that $P=d\log\ell_x([\lambda])$.
Then by definition, $P$ is contained in $F^*_{|\lambda|}$, and since $P$ is an interior point of $F$, the entire $F$ must be contained in $F^*_{|\lambda|}$.

If $F^*=F^*_{|\lambda|}$, then we are done.
Otherwise, it remains to show that for every interior point $P'=d\log\ell_x([\mu])$ of $F^*$, we have $|\mu|=|\lambda|$.
Since $F^*$ is contained in $F^*_{|\lambda|}$ which consists of the images of projective measured laminations whose supports are contained in $|\lambda|$, we have $|\mu| \subset |\lambda|$.
As shown above, since $P'$ is an interior point, $F^*$ is contained in $F^*_{|\mu|}$.
By applying the same argument by exchanging the roles of $P$ and $P'$, we have $|\lambda| \subset |\mu|$.
Thus we see that $|\lambda|=|\mu|$, and we have completed the proof.
\end{proof}

By \cref{boundary face}, for any face $F^*$ of $\mathcal S_x^*$, there is a unique (unmeasured) geodesic lamination $\nu$ such that any interior point of $F^*$ is the image $d\log\ell_x([\lambda])$  for some projective measured lamination $[\lambda]$ with $|\lambda|=\nu$.
We call such a geodesic lamination $\nu$ the {\em interior lamination}  of $F^*$.
Now, we can refine \cref{codimension of face} to the following corollary without changing the proof.

\begin{corollary}
\label{refined codimension}
Let $F$ be a face whose interior lamination is $\nu$.
Then we have $\codim(F)=6g-6-C(\nu).$
\end{corollary}

As was shown in \cite[Theorem 6.7]{HOP}, any interior point of a face $F^*_{|\lambda|}$ is the image of a projective measured lamination whose support is equal to $|\lambda|$.
On the other hand, we also remarked in the same paper that the opposite implication fails to hold.
In the following proposition, we characterise boundary points of  $F^*_{|\lambda|}$ which are the images of a projective laminations with the same supports as $|\lambda|$.
For the support $|\lambda|$ of a measured lamination $\lambda$, we denote by $M(|\lambda|)$ the space of the transverse measure on $|\lambda|$, and $PM(|\lambda|)$ its projectivisation.

\begin{proposition}
Let $P$ be a point on the boundary of the face $F^*_{|\lambda|}$ corresponding to a projective measured lamination $[\lambda]$.
Suppose that $P=d\log\ell_x([\mu])$ for a projective measured lamination $[\mu]$ with  $|\mu|=|\lambda|$.
Then there is a component $\lambda_0$ of $\lambda$ which is not uniquely ergodic, and the corresponding component $\mu_0$ of $\mu$ has a transverse measure lying on the boundary of $PM(|\mu_0|)$.
\end{proposition}
\begin{proof}
By definition, the face $F^*_{|\lambda|}$ consists of the images of projective measured laminations whose supports are contained in $|\lambda|$.
Since $d\log\ell_x$ is a homeomorphism onto $\mathcal S_x^*$, its restriction to $PM(|\lambda|)$ gives a homeomorphism onto $F^*_{|\lambda|}$.
Any boundary point of $PM(|\lambda|)$ either has measure $0$ on a component or corresponds to a projective transverse measure on a non-uniquely ergodic component.
Thus we have completed the proof.
\end{proof}

In \cite[Theorem 6.7]{HOP}, it was proved that $F^*_{|\lambda|}$ is an exposed face.
We shall next prove that no other faces are exposed.

\begin{theorem}
\label{exposed}
A face of $\mathcal S_x^*$ is exposed if and only it is equal to $F^*_{|\lambda|}$ for some projective lamination $[\lambda]$.
\end{theorem}
\begin{proof}
The \lq if ' part was already proved in \cite[Theorem 6.7]{HOP}.
We shall prove the \lq only if' part.
Let $F^*$ be a face of $\mathcal S_x^*$, and suppose that it does not have the form of $F^*_{|\lambda|}$.
Then by \cref{boundary face}, $F^*$ is contained in $F^*_{|\lambda|}$ such that the interior of $F^*$ consists of the images of projective measured laminations whose supports are equal to $|\lambda|$.
Let $v$ be a point on $\mathcal S_x$ which is orthogonal to $F^*$.
Since $v$ infinitesimally stretches the interior points of $F^*$, which corresponds to images of projective measured laminations whose supports are equal to $|\lambda|$,  it lies in $F_{|\lambda|}$.
Therefore $v$ is also orthogonal to $F^*_{|\lambda|}$, which properly contains $F^*$.
Thus we see that $F^*$ cannot be an exposed face.
%
%
%
%
\end{proof}

It is immediate from the definitions that for any face $F^*$ of $\mathcal S^*_x$, $\dim(F^*)+\codim(F^*) \leq 6g-7$.
In \cite[\S6.6]{HOP}, it was proved that the sum of dimension and codimension is equal to $6g-7$ for faces in $\mathcal S^*_x$ corresponding to projective classes of  weighted multi-curves.
Here we shall prove that this is indeed a necessary and sufficient condition.

\begin{theorem}
\label{sum}
For a projective measured lamination $[\lambda]$ the following two conditions are equivalent.
\begin{enumerate}[(1)]
\item $\dim F^*_{|\lambda|}+\codim F^*_{|\lambda|}=6g-7$.
\item $\lambda$ is a weighted multi-curve.
\end{enumerate}
\end{theorem}
\begin{proof}
Since the implication from (2) to (1) was already proved in \cite{HOP}, we have only to show that (1) implies (2).
Let $\tau$ be a recurrent train track adapted to $\lambda$.
Since $\lambda$ is not a weighted multi curve, we can find a component $\lambda_0$ which is not a simple closed curve and a component $\tau_0$ of $\tau$ adapted to $\lambda_0$.

Then it is easy to see that  $\tau_0$ carries measured laminations whose supports are not equal to $|\lambda_0|$: indeed  rational weights carry weighted multi-curves, whose supports cannot be $|\lambda_0|$.
Therefore $\dim M(|\lambda_0|) < C(|\lambda_0|)$, and hence $\dim M(|\lambda|) < C(|\lambda|)$ and $PM(|\lambda|) < C(|\lambda|)-1$.
It follows that $\dim(F^*_{|\lambda|})+\codim(F^*_{|\lambda|})=\dim PM(|\lambda|)+6g-6-C(|\lambda|)<6g-7.$
\end{proof}

Now, we return to considering the unit tangent sphere $\mathcal S_x$.
We shall describe  codimensions of  faces of $\mathcal S_x$ in terms of  corresponding chain-recurrent geodesic laminations.
For a chain recurrent geodesic lamination $\lambda$, recall that the union of its minimal components is denoted by  $\lambda^c$.
Then, we can express codimensions of faces as follows.

\begin{theorem}
Let $F$ be a face of $\mathcal S_x$.
Let $\lambda$ be a chain-recurrent geodesic lamination such that $F=F_{\lambda}$, whose existence is guaranteed by \cref{face}.
Then we have $\codim(F)=\dim PM(\lambda^c)$,
\end{theorem}
\begin{proof}
Let $w=d\log\ell_x([\mu])$ be a point in $\mathcal S^*_x$ orthogonal to $F$.
Since $F$ contains every stretch vector with respect to a maximal chain-recurrent geodesic lamination containing $\lambda$, the lamination $\mu$ must be contained in $\lambda$ by \cite[Lemma 2.7]{BOP}.
Since $\mu$ is a measured lamination, this implies it is contained in $\lambda^c$.

Conversely if $\mu$ is contained in $\lambda^c$, it is infinitesimally stretched most by any stretch vector in $F_\lambda$, which corresponds to a maximal chain-recurrent geodesic lamination containing $\lambda$.
Therefore, the co-vectors orthogonal to $F$ constitute the set $\{d\log\ell_x([\mu])) \in \mathcal S^*_x\mid |\mu| \subset \lambda^c\}$, and its dimension is equal to $\dim PM(\lambda^c)$.
\end{proof}

\bibliography{dim.bib}
\bibliographystyle{acm}
\end{document}